\documentclass[11pt]{amsart}
\def\uglyO{y}
\def\draft{n}
\usepackage{amssymb}
\usepackage{graphicx} 
\usepackage{subfig}
\usepackage{dbnsymb}
\usepackage{enumerate}
\usepackage{mathtools}
\usepackage{color,soul}
\usepackage{tikz}
\usetikzlibrary{matrix}
\usepackage{caption}
\usepackage{amsmath, amsthm}
\usepackage{mathtools}
 \usepackage{relsize}
 \usetikzlibrary{cd}
 \usetikzlibrary{decorations.pathreplacing}
 \usepackage{tikz,calc}
\usepackage{color}
\usepackage[margin=1.2in]{geometry}

\usepackage{multicol}
\usepackage{comment}

\usepackage{wrapfig}
\usepackage{cutwin}

\usepackage{lmodern}
\usepackage{enumitem}
\usepackage{lipsum}
\usepackage{stackengine}
\usepackage{appendix}
\usepackage{scrextend}
\usepackage[breaklinks=true]{hyperref}\hypersetup{colorlinks,
  linkcolor={green!50!black},
  citecolor={green!50!black},
  urlcolor=blue
}
\newtheorem{thm}{Theorem}[section]

\newtheorem{prop}[thm]{Proposition}
\newtheorem{lem}[thm]{Lemma}

\newcommand{\theoremname}{Theorem:}

  \theoremstyle{definition}

  \newtheorem*{claim*}{Claim}

  \newtheorem*{question*}{Question}
  \newtheorem*{answer*}{Answer}
  \newtheorem*{application*}{Application}

  \newenvironment{dproof}[2] {\paragraph{\emph{Proof of {#1} {#2}.}}}{\hfill$\square$}

  \theoremstyle{remark}
  
  \newtheorem*{rmk*}{Remark}

\usepackage{bbm}
\newcommand{\Q}{\mathbb{Q}}
\newcommand{\R}{\mathbb{R}}
\newcommand{\Z}{\mathbb{Z}}

\def\tilO{{\tilde{O}}}

\def\draftcut{\if\draft y \cleardoublepage \fi}
\def\ifugly#1#2{\if\uglyO y{#1}\else{#2}\fi}
\def\ifug#1{\ifugly{\tilO(#1)}{\sim #1}}

\begin{document}

\title{Computing finite type invariants efficiently}

\author{Dror Bar-Natan}
\address{University of Toronto}
\email{drorbn@math.toronto.edu}
\urladdr{\url{http://www.math.toronto.edu/drorbn}}

\author{Itai Bar-Natan}
\address{University of California, Los Angeles}
\email{itaibn@math.ucla.edu}
\urladdr{}

\author{Iva Halacheva}
\address{Northeastern University}
\email{i.halacheva@northeastern.edu}
\urladdr{\url{https://sites.google.com/site/ivahalacheva3/}}

\author{Nancy Scherich}
\address{Elon University}
\email{nscherich@elon.edu}
\urladdr{\url{http://www.nancyscherich.com}}

%\subjclass{57M25}

\thanks{}

\begin{abstract}
We describe an efficient algorithm to compute finite type invariants of type $k$ by first creating, for a given knot $K$ with $n$ crossings, a look-up table for all subdiagrams of $K$ of size $\lceil \frac{k}{2}\rceil$ indexed by dyadic intervals in $[0,2n-1]$. Using this algorithm, any such finite type invariant can be computed on an $n$-crossing knot in time $\ifug{n^{\lceil\frac{k}{2}\rceil}}$, a lot faster than the previously best published bound of $\ifug{n^k}$.
\end{abstract}

\maketitle

%To-do:
%\begin{itemize}
%    \item Rewrite the background (to make it different from paper 1)\textcolor{red}{I did a bit of this. it is enough?}
%\end{itemize}

\section{Introduction}

Finite type invariants, also known as Vassiliev invariants  \cite{Vassiliev90, Vassiliev}, underlie many of the classical knot invariants, for instance they include the coefficients of the Jones, Alexander, and more generally HOMFLY-PT polynomials \cite{BL93, BN1}.
A knot invariant $\zeta$ is said to be of \emph{finite type $k$} (equivalently of degree $k$) if it vanishes on all knots with at least $k+1$ double points, where $\zeta$ is extended to knots with double points by the formula:
$$\zeta(\doublepoint)=\zeta(\overcrossing) - \zeta(\undercrossing).$$
For example, the linking number of a two-component knot is a finite type invariant of type 1. 
In our main theorem, we provide an algorithm to compute finite type invariants from a planar projection of a knot in a surprisingly efficient time depending on the crossing number of the knot diagram.

For complexity measurements in this paper, we care only about polynomial degree  (i.e.  we ignore constants and $\log (n)$ terms).
\ifugly{In this vein, recall that $f(n)=\tilO(g(n))$ means that there exist positive numbers $c,p,N$ so that for all $n>N$, we have that $f(n)<cg(n)(\log g(n))^p$.}{We write $f(n)\sim g(n)$ to mean that there exist positive numbers $c,p,N$ so that for all $n>N$, we have that 
\[\frac{1}{c}g(n)(\log(n))^{-p}<f(n)<cg(n)(\log(n))^p.\]
For example, for us, $n^4\sim {5.4}\; {n^4}(\log(n))^8$.} 

\vspace{.2cm}

\noindent \textbf{Main Theorem.} \textit{Finite type invariants of type $k$ can be computed on an $n$-crossing knot in \ifugly{time $\tilO(n^{\lceil k/2\rceil})$.}{time at most $\sim n^{\lceil k/2\rceil}$.}}
\vspace{.2cm}

The result is a surprising result, as before this theorem the fastest algorithm (known to the authors) to compute a type $k$ invariant on a knot diagram with $n$ crossings took time $\ifug{n^k}$ \cite{BBHS}, and it was commonly believed that this was the fastest possible. 
There are specific finite type invariants which can be computed much faster, such as the linking number and the coefficients of the Alexander polynomial, but these are special cases.
The Main Theorem gives the current fastest known algorithm that works for \emph{all} finite type invariants, and shows that the computational time can be reduced to roughly the square root of the previously fastest known algorithm.
In a previous paper \cite{BBHS}, we proved that finite type invariants can be computed efficiently using 3D methods. We argued that most knot invariants, including finite type invariants, should be more efficiently computed using 3D methods rather than 2D methods.
However, the Main Theorem is significant as it is a 2D method that currently outperforms all known 3D methods to compute finite type invariants.

%Additionally, Theorem \ref{thm:MainResult} is significant as the proof 
%That is, given a \emph{grid knot} with volume $V$, a finite type invariant of type $k$ can be computed in time at most $\sim V^k$.
%In this paper, we provide a $2D$ algorithm which is more efficient than the previous 3D algorithm, and when measured in terms of volume this new algorithm can be computed in time $\sim V^{\frac{2}{3}k}$.

For impatient readers, the key formula in this paper is Equation (\ref{eqn:Finalformula}). The preceding pages include the definitions leading up to the formula, and the proof that the formula can be evaluated in time $\ifug{n^{\lceil k/2 \rceil}}$.
The proof uses a space-time tradeoff and the Main Theorem can be generalized to highlight the space requirements.

\vspace{.2cm}
\noindent \textbf{Space-Aware Version of the Main Theorem.} \textit{For any integer $f\leq\frac{k}{2}$, any finite type invariant of type $k$ can be computed on an $n$-crossing knot in time $\ifug{n^{k-f}}$ using storage space \ifugly{$\tilO(n^f)$}{at most $\sim{n^f}$}.}
\vspace{.2cm}

%For example, when comparing sizes or complexities, we agree that $7n^3(\log n)^5$ is the same as $n^3$ and as $\frac{1}{120}n^3/\log n$, and write $7n^3(\log n)^5 \sim n^3 \sim \frac{1}{120}n^3/\log n$.

\vspace{2mm}
\noindent \textbf{Acknowledgements.}
 The first author was supported by NSERC-RGPIN-2018-04350 and by the Chu Family Foundation (NYC). The third author was supported by the Natural Science Foundation Grant No. DMS-2302664. This material is also based upon work supported by the National Science Foundation under Grant No. DMS-1929284 while the fourth author was in residence at the Institute for Computational and Experimental Research in Mathematics in Providence, RI, during the Braids program. 
 This project is partially sponsored by the Provost Office of Elon University.
 We would like to thank ICERM for hosting the first and third authors for a week long visit. Finally we thank A.~Referee for their extremely valuable comments.

\section{Background}

\subsection{Gauss diagrams}

A \emph{Gauss diagram} of an $n$-crossing long knot diagram parametrized by an open interval $I_\R\subseteq\mathbb{R}$ is given by the interval $I_\R$ along with $n$ decorated arrows (equivalently, oriented perfect matchings of $2n$ points, where the arrows are the edges of the matching). Usually $I_\R=(-1,2n)$ but $I_\R$ can be any open interval.  
To distinguish between intervals of real numbers  and integers, we use the notation of a subscript of $\R$ or $\Z$ on the interval. So $I_\Z=I_\R\cap\Z$.
Each arrow corresponds to one of the $n$ crossings of the knot and has endpoints at the integer values in $I_\R$ (the endpoints are in $I_\Z)$. 
The head of an arrow is at the point in $I_\Z$ which parametrizes the lower strand of the crossing and the tail of the arrow is at the point which parametrizes the upper strand of the crossing (we assume all crossing points occur at the integer location of the paramaterization). 
Each arrow is decorated with a sign corresponding to the sign of the crossing. We say that such a diagram is \emph{numbered} by $I_\Z$, where usually $I_\Z=[0,2n)_\Z$.
Figure \ref{fig:gaussdiagrams} (A) shows an example of a Gauss diagram.
For a Gauss diagram $D$, the quantity $|D|$ is the number of arrows in $D$.
Let $\mathcal{GD}=\langle$Gauss diagrams$\rangle$ denote the $\Z$-module of $\Z$-linear combinations of Gauss diagrams, and let $\mathcal{GD}_k$ denote the subspace spanned by Gauss diagrams with $k$ or fewer arrows.

\begin{figure}[]

\centering
\begin{picture}(380,270)
\put(28,-4){\includegraphics[width=.7\textwidth]{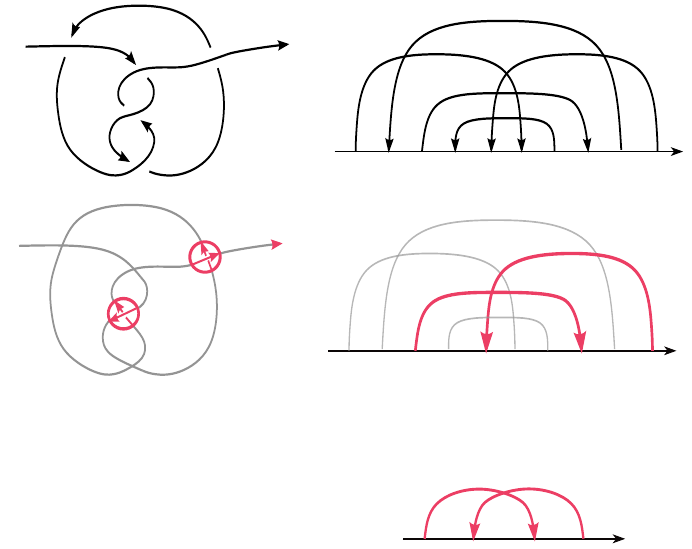}}
\put(30,218){\textcolor{magenta}{\textbf{-1}}}
\put(54,217){\textcolor{magenta}{\textbf{0}}}
\put(97,198){\textcolor{magenta}{\textbf{1}}}
\put(83,185){\textcolor{magenta}{\textbf{2}}}
\put(94,154){\textcolor{magenta}{\textbf{3}}}
\put(111,219){\textcolor{magenta}{\textbf{4}}}
\put(56,207){\textcolor{magenta}{\textbf{5}}}
\put(86,163){\textcolor{magenta}{\textbf{6}}}
\put(72,192){\textcolor{magenta}{\textbf{7}}}
\put(85,205){\textcolor{magenta}{\textbf{8}}}
\put(122,211){\textcolor{magenta}{\textbf{9}}}
\put(156,217){\textcolor{magenta}{\textbf{10}}}
%--
\put(185,163){\textcolor{magenta}{\textbf{0}}}
\put(200,163){\textcolor{magenta}{\textbf{1}}}
\put(215,163){\textcolor{magenta}{\textbf{2}}}
\put(228,163){\textcolor{magenta}{\textbf{3}}}
\put(243,163){\textcolor{magenta}{\textbf{4}}}
\put(260,163){\textcolor{magenta}{\textbf{5}}}
\put(275,163){\textcolor{magenta}{\textbf{6}}}
\put(290,163){\textcolor{magenta}{\textbf{7}}}
\put(305,163){\textcolor{magenta}{\textbf{8}}}
\put(317,163){\textcolor{magenta}{\textbf{9}}}
%--

\put(210,73){\textcolor{magenta}{\textbf{2}}}

\put(241,73){\textcolor{magenta}{\textbf{4}}}

\put(286,73){\textcolor{magenta}{\textbf{7}}}

\put(316,73){\textcolor{magenta}{\textbf{9}}}
%--
\put(185,207){\Large -}
\put(250,235){\Large -}
\put(230,200){\Large -}
\put(240,187){\Large -}
\put(315,207){ +}
%---
\put(230,110){\Large -}
\put(315,112){ +}
%--
\put(145,185){$\longleftrightarrow$}
\put(145,100){$\longleftrightarrow$}
%---
\put(345,175){(A)}
\put(345,80){(B)}
\put(345,0){(C)}
\put (250, 60) {\rotatebox{-90}{ $\longrightarrow$}}
\put(260, 45){$\psi$}
%---
\put(214,-8){\textcolor{magenta}{\textbf{0}}}

\put(236,-8){\textcolor{magenta}{\textbf{1}}}

\put(265,-8){\textcolor{magenta}{\textbf{2}}}

\put(286,-8){\textcolor{magenta}{\textbf{3}}}

\put(220,20){\Large -}
\put(275,20){ +}
\end{picture}
\caption{(A) An example of the Gauss diagram of a long knot diagram.  (B) A $2$-arrow subdiagram of a Gauss diagram. (C) The forgetful map $\psi$ applied to a subdiagram yielding a renumbered subdiagram.}
\label{fig:gaussdiagrams}

\end{figure}

Let $D$ be a Gauss diagram with $n$ arrows numbered by $[0,2n)_\Z$. 
A \emph{$k$-arrow subdiagram} of  $D$ is a diagram consisting of $I_\R$ and a subset of $k$ decorated arrows from $D$.  
A subdiagram corresponds to a choice of $k$ crossings in the knot diagram represented by the Gauss diagram $D$. 
An example is shown in Figure \ref{fig:gaussdiagrams}(B).
Notice that a subdiagram $D'$ of $D$ keeps the original numbering along the interval $I_\R$ and the $2k$ endpoints of $D'$ will be spread out amongst the $2n$ points in $I_\R$.
A $k$-arrow subdiagram of $D$ is \emph{not} a proper Gauss diagram because of this numbering issue.
To make a $k$-arrow subdiagram $D'$ into a Gauss diagram, we can apply the forgetful map $\psi$ to $D'$ which renumbers $[0,2n)_\Z$ monotonically to $[0,2k)_\Z$, in essence forgetting how $D'$ was realized as a subdiagram of $D$. 
We will call $\psi(D')$ a \emph{renumbered subdiagram} of $D$.
An example is shown in Figure \ref{fig:gaussdiagrams}(C).

Given two disjoint subdiagrams $E$ and $F$ of a Gauss diagram $D$, with $e=|E|$ and $f=|F|$, we say that the \emph{pattern} $P=\pi_D(E,F)$ of $E$ and $F$ within $D$ is the $2e$-element subset of $[0,2(e+f))_\Z$ numbering the ends of the images of the arrows of $E$ within $\psi(E\cup F)$ (and so the complementary subset $[0,2(e+f))_\Z\setminus P$ numbers the arrow ends of $F$). In the special case where $D=E\cup F$, we simply write $P=\pi(E,F)$.

Conversely, if $E$ and $F$ are Gauss diagrams with $e=|E|$ and $f=|F|$ and $P$ is an appropriately-sized \emph{pattern} (namely, $P\subset[0,2(e+f))_\Z$ and $|P|=2e$), then there is a unique Gauss diagram $D$ with $|D|=e+f$ containing $E$ and $F$ as disjoint subdiagrams such that $\pi_D(E,F)=P$. We denote $D=E\#_PF$ and say that $D$ is the \emph{superimposition} of $E$ and $F$ with pattern $P$. We extend $\#_P$ to a bilinear map $\#_P:\mathcal{GD}_e\times \mathcal{GD}_f\rightarrow \mathcal{GD}_{e+f}$. Two examples are shown in Figure~\ref{fig:superimposition}.

\begin{figure}[]
\scalebox{.85}{
\centering
\begin{picture}(320,160)
\put(-40,10){\includegraphics[width=.9\textwidth]{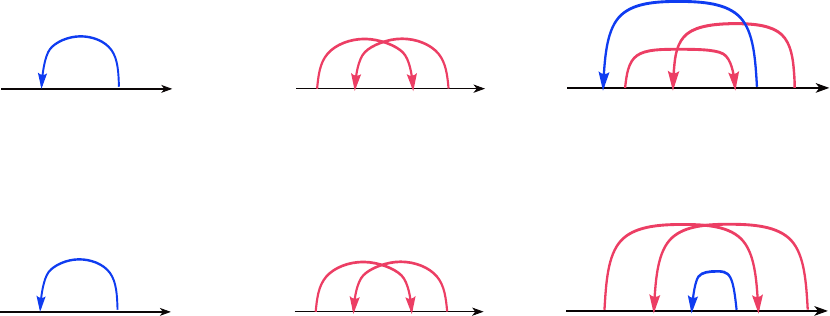}}

\put(-23,2){\textcolor{blue}{\textbf{0}}}
\put(12,2){\textcolor{blue}{\textbf{1}}}

\put(110,2){\textcolor{red}{\textbf{0}}}
\put(128,2){\textcolor{red}{\textbf{1}}}
\put(155,2){\textcolor{red}{\textbf{2}}}
\put(170,2){\textcolor{red}{\textbf{3}}}

\put(43,40){\Large $P_2=\{2,3\}$}

\put(60,10){\huge $\#_{P_2}$}
\put(205,10){\huge $=$}

\put(247,2){0}
\put(270,2){1}
\put(285,2){2}
\put(306,2){3}
\put(320,2){4}
\put(340,2){5}
%--
\put(-23,108){\textcolor{blue}{\textbf{0}}}
\put(12,108){\textcolor{blue}{\textbf{1}}}

\put(110,108){\textcolor{red}{\textbf{0}}}
\put(128,108){\textcolor{red}{\textbf{1}}}
\put(155,108){\textcolor{red}{\textbf{2}}}
\put(170,108){\textcolor{red}{\textbf{3}}}

\put(43,146){\Large $P_1=\{0,4\}$}

\put(60,116){\huge $\#_{P_1}$}
\put(205,116){\huge $=$}

\put(242,108){0}
\put(255,108){1}
\put(278,108){2}
\put(306,108){3}
\put(318,108){4}
\put(335,108){5}
%--
%\put(-60,140){(A)}
%\put(-60,10){(B)}

\end{picture}
}
\caption{Two examples of superimposing the same diagrams along different patterns.}
\label{fig:superimposition}

\end{figure}

%Gauss diagrams play an important role in the theory of finite type invariants. Any knot invariant $V$ taking numerical values can be extended to an invariant of knots with finitely many double points (i.e. immersed circles whose only singularities are transverse self-intersections), using the following locally defined equation:
%$$V(\doublepoint)=V(\overcrossing) - V(\undercrossing)$$
%The above should be interpreted in the setting of knot diagrams which coincide outside of the given crossing.
%A knot invariant is a \textit{finite type invariant of type $k$} if it vanishes on all knot diagrams with at least $k+1$ double points \cite{BL93,  BN1}.

We denote by $\varphi_k : \{$knot diagrams$\}\rightarrow \mathcal{GD}_k$ the map which sends a knot diagram to the sum of all of the \emph{renumbered} subdiagrams of its Gauss diagram which have exactly $k$ arrows. Note that $\varphi_k=\psi\circ \bar{\varphi}_k$ where $\bar{\varphi}_k$ sends a knot diagram to the sum of all of the subdiagrams of its Gauss diagram which have exactly $k$ arrows, and $\psi$ renumbers each summand to make it a Gauss diagram.
Let $\varphi_{\leq k}:=\sum_{i=1}^k\varphi_i$ be the map that sends a knot diagram to the sum of all of the renumbered subdiagrams of its Gauss diagram which have \emph{at most} $k$ arrows. 
The maps $\varphi_{k}$ and $\varphi_{\leq k}$ are \emph{not}  invariants of knots but every finite type invariant factors through $\varphi_{\leq k}$, as follows from the next theorem.

\begin{thm}[Goussarov-Polyak-Viro \cite{GPV}, see also \cite{Roukema}]\label{thm:FactorsThruPhi} %https://arxiv.org/abs/math/9810073 
A $\Q$-valued knot invariant $\zeta$ is of type $k$ if and only if there is a linear functional $\omega$ on $\Q\otimes\mathcal{GD}_k$ such that $\zeta=\omega\circ\varphi_{\leq k}$.
\end{thm}

We show that $\varphi_{\leq k}$ can be computed in time $\ifug{n^{\lceil k/2 \rceil}}$. This result, combined with the above theorem, proves that all finite type invariants of type $k$ can be computed in time $\ifug{n^{\lceil k/2 \rceil}}$, which is the main result of this paper.

It is surprising that $\varphi_{\leq k}$ can be computed in time $\ifug{n^{\lceil k/2 \rceil}}$ because, at first glance, it would seem that one must require \ifugly{time $\tilO(n^k)$.\footnote{We slightly abuse notation and use $\tilO$ also to mean ``equal up to constants and logarithms''.}}{time $\sim n^k$.} A Gauss diagram with $n$ arrows has $\sum_{i=1}^k{n\choose i}$ subdiagrams with $k$ or fewer arrows. Because \ifugly{${n\choose i}=\tilO(n^i)$}{${n\choose i}\sim n^i$}, a Gauss diagram with $n$ arrows has \ifugly{$\sum_{i=1}^k$ ${n\choose i}=\tilO\left(\sum_{i=1}^k n^i\right)=\tilO(n^k)$}{$\sum_{i=1}^k$ ${n\choose i}\sim \sum_{i=1}^k n^i\sim n^k$} subdiagrams with $k$ or fewer arrows. 
So $\bar{\varphi}_{\leq k}$ evaluated on a knot diagram with $n$ crossings will be a sum of $\ifug{n^k}$ subdiagrams.
Note that the outputs of $\bar{\varphi}_k$ and  $\varphi_k$ have the same number of summands, but $\varphi_k$ will have repeated terms and $\bar{\varphi}_k$ will not.
So where do the computational savings come from?
The idea is to break the computation of $\varphi_{\leq k}$ into two parts, one of which can be quickly pre-computed in a look-up table. The creation of this look-up table uses counting techniques taking advantage of dyadic intervals. These techniques are completely self-contained, and unrelated to finite type invariants and knot theory. In the next section, we describe these techniques that we will apply to prove that that $\varphi_{\leq k}$ can be computed in time $\ifug{n^{\lceil k/2 \rceil}}$.

\section{Computational Preliminaries: Counting Techniques using Dyadic intervals}

\subsection{Counting with a look-up table.}
For the purpose of this paper, a \emph{look-up table} is a lexicographically-ordered list of (key $\mapsto$ value) entries, or more formally a lexicographically-ordered associative array. 
Below, Theorem \ref{thm:table1} shows how to use a look-up table to count elements of a set inside $[n]^\ell$, where $[n]:=[1,n]_\Z$. While we will need a generalized version of this theorem, the proof of Theorem \ref{thm:table1} showcases nicely how dyadic intervals are used to get computational savings. 

\begin{thm}\label{thm:table1}
    Let $Q$ be an enumerated subset of $[n]^\ell$ with $|Q|=q=n^f$, where $0\leq f\leq\ell$ (typically $0<f<\ell$ and so $Q$ is ``big'' yet ``much smaller'' than $[n]^\ell$). In time $\ifug{q}$, a look-up table of size $\ifug{q}$ can be created so that computing $|Q\cap R|$ will take time $\ifug{1}$ for any rectangular box $R\subset [n]^\ell$.
\end{thm}

The straightforward approach to this theorem would be to create a look-up table with key-value pairs of the form $(R \mapsto |Q\cap R|)$ for all possible rectangular boxes $R$. A rectangular box in $[n]^\ell$ is determined by choosing two interval endpoints in each coordinate of $[n]^\ell$, so there are $\ifug{n^{2\ell}}$ possible rectangular boxes in $[n]^\ell$. This look-up table would take much longer than $\ifug{q}$ to create.
The trick is to create a restricted look-up table using only rectangular boxes which intersect $Q$ non-trivially and whose sides are dyadic intervals. The structure is reminiscent of the data structures quadtrees / octrees which are used in computer graphics.

To prove Theorem \ref{thm:table1}, let us first start with some preliminaries on dyadic intervals.

\subsection{Dyadic intervals}\label{sec:dyadic}

For the purpose of this paper, a \emph{dyadic interval} is a half-closed interval of the form $\left[2^pq,2^{p}(q+1)\right)_\Z$ for $p,q \in \Z_{\ge 0}$.
A dyadic interval can be expressed via a binary expansion up to a certain accuracy, i.e. a fixed sequence of $0$'s and $1$'s followed by a fixed number of $*$'s or `free' entries, and the integers contained in the interval are all possible completions of the expansions, i.e. all possible ways of replacing the $*$'s with $0$'s and $1$'s.
For example $[2^22,2^23)_\Z=\{8,9,10,11\}=\{1000_2,1001_2,1010_2,1011_2\}$, expressed in decimal and binary expansion respectively. Using the binary expansion, every number in the interval is of the form $10\!*\!*\!$ where the least significant two bits are free and the most significant two bits are fixed to be 10.
We write $u$ for a dyadic interval where $u$ is a binary sequence followed by some number of $*$'s.

A dyadic interval $u$ is \emph{maximal} in  $(b,c)_\Z$ if $u\subset (b, c)_\Z$ and any larger dyadic interval containing $u$ is not contained in $(b, c)_\Z$. 
Notice if two dyadic intervals overlap, then one is contained in the other. 
%This can be seen as follows: if  $2^jk\leq 2^lm< 2^j(k+1)$ then $k\leq 2^{l-j}m< (k+1)$ and thus $k=2^{l-j}m$. So the two left endpoints  of the intervals are equal and one must be a subset of the other. 
Therefore, the maximal dyadic intervals of $(b,c)_\Z$ are disjoint.
Every interval decomposes uniquely as a disjoint union of maximal dyadic intervals, as described in the following elementary Lemma.

\begin{lem}\label{lem:dyadic}
For any interval $(b,c)_\Z\subset \Z_{\ge 0}$, there are at most $2 \log_2(c-b)$ maximal dyadic intervals contained in $(b,c)_\Z$, and $(b,c)_\Z$ is a disjoint union of its maximal dyadic intervals.
\end{lem}

For a binary number, a truncation process can be used to compute dyadic intervals containing that binary number.\\

\noindent \textbf{\emph{Truncation process}}: Given a binary number $x$ with $m$ bits of accuracy, starting from right to left, replacing one bit at a time with a $*$ generates a list of $m$ dyadic intervals which contain $x$.
For example, let $x=7=00111_2$ where $m=5$, then the truncation process generates the list of dyadic intervals $\{00111,0011\!*,001\!*\!*,00\!*\!**,0\!*\!*\!*\!*\}$.
If a number $x$ (in decimal notation) is at most $\ell$, then the number of bits needed  to describe $x$ in binary expansion is at most $\log_2(\ell)$, and so $x$ is contained in at most $\log_2(\ell)$ dyadic intervals coming from the truncation process with $\log_2(\ell)$ bits of accuracy.

\subsection{Proof and generalization of Theorem \ref{thm:table1}}

We proceed with the proof of Theorem~\ref{thm:table1}, as well as its generalization, Proposition~\ref{prop:table2}, which we will use in the proof of the main theorem. 
\vspace{0.3cm}

\begin{dproof}{Theorem}{\ref{thm:table1}}
A dyadic rectangular box in $[n]^\ell$ is a product of $\ell$ dyadic intervals of $[n]$.
    We describe a process to create a restricted look-up table of key-value pairs of the form $(R_d \mapsto |Q\cap R_d|)$ for only $R_d$ that are dyadic rectangles in $[n]^\ell$ and such that $|Q\cap R_d|>0$.

    To create the table, run through the enumerated elements $x\in Q$. For each $x$, from the truncation process described in Section \ref{sec:dyadic}, $x$ is contained in at most $(\log_2(n))^\ell$ dyadic rectangular boxes $R_d$ in $[n]^\ell$. In the table, increment the value for $R_d$ by 1 for each such $R_d$ containing $x$, or create such an entry if it didn't already exist. 
    Since there are $q$ elements in $Q$, creating this table takes time $\ifugly{\tilO\left(q (\log_2(n))^\ell\right)=\tilO(q)}{\sim q (\log_2(n))^\ell\sim q}$, and there are $\ifug{q}$ elements in the table. 
    Using standard binary sorting techniques, accessing and modifying values in the table takes time $\ifugly{\tilO(\log_2 q)=\tilO(1)}{\sim \log_2 q\sim 1}$.

    A general non-dyadic rectangular box $R$ in $[n]^\ell$ is the disjoint union of at most $( 2\log_2(n))^\ell$ $\ifugly{=\tilO(1)}{\sim 1}$ maximal dyadic rectangular boxes, by Lemma \ref{lem:dyadic}. To count $|Q\cap R|$, one retrieves the values in the look-up table with key each maximal dyadic rectangular box in $R$, and then sums together all those stored values to get $|Q\cap R|$.
    Since retrieval takes time $\ifug{1}$ and the sum is over $\ifug{1}$ elements, after the look-up table is completed, it takes time $\ifug{1}$ to compute $|Q\cap R|$.
\end{dproof}

Viewing $|Q\cap R|$ as the sum $\displaystyle{\sum_{x\in Q\cap R}}1$, Theorem \ref{thm:table1} can be generalized, using almost exactly the same proof, to compute weighted sums $\displaystyle{\sum_{x\in Q\cap R}}\theta_x$ where the weights $\theta_x$ are valued in some $\Z$-module.
This generalization is stated below. It is the version of Theorem~\ref{thm:table1} that will be used to prove the main result in Section \ref{sec:mainresult}.

\begin{prop}\label{prop:table2}
    Let $Q$ be an enumerated subset of $[n]^\ell$ with $|Q|=q=n^f$, where $0\leq f\leq\ell$. Let $M$ be a free $\Z$-module of rank $\ifug{1}$, and let $\theta:[n]^\ell\rightarrow M$ be a map that is zero outside of $Q$.
     Then, in time $\ifug{q}$, a look-up table of size $\ifug{q}$ can be created so that computing $\sum_R \theta $ will take time $\ifug{1}$ for any rectangular box $R\subset [n]^\ell$.\footnote{There is a small caveat to Proposition~\ref{prop:table2} which requires that in some basis for $M$, the coefficients of $\theta$ must be computable in time $\ifug{1}$ and their size must be $\ifug{1}$, which may not always be the case. However, for the purposes of this paper and the application for which this proposition is used, this issue does not arise and so we chose to leave this technical detail out of the statement. }
\end{prop}

\section{Main Result}\label{sec:mainresult}

In this section, we state and prove the main theorem of this article. 
The strategy to quickly compute $\varphi_k$ on a diagram $D$ is to view a $k$-arrow subdiagram of $D$ as the superimposition of two smaller subdiagrams $E$ and $F$.
The subdiagram $E$ is placed inside $D$ first, and instead of placing $F$, a look-up table is created to count in how many ways $F$ could have been placed. 
Using Proposition~\ref{prop:table2}, this look-up table can be computed and accessed very quickly, which ultimately gives the computational savings for the final result.

\vspace{.4cm}

\noindent \textbf{Main Theorem.} \textit{Finite type invariants of type $k$ can be computed on an $n$-crossing knot in time \ifugly{$\tilO(n^{\lceil k/2\rceil})$}{at most $\sim n^{\lceil k/2\rceil}$}.}
\vspace{.2cm}

%\begin{thm}\label{thm:MainResult}
    
 %Finite type invariants of type $k$  can be computed on a knot with $n$ crossings in time $\sim n^{\lceil k/2\rceil}$. 
%\end{thm}

\begin{proof}
By Theorem \ref{thm:FactorsThruPhi}, it suffices to show that $\varphi_{\leq k}$ can be computed in time $\ifug{n^{\lceil k/2\rceil}}$.
Since $\varphi_{ \leq k}=\sum_{i=1}^k \varphi_{i}$,
it suffices to prove that $\varphi_k$ can be computed in time $\ifug{n^{\lceil k/2\rceil}}$. This shows that $\varphi_{\leq k}$ can be computed in time \ifugly{$\tilO(k\cdot n^{\lceil k/2\rceil})=\tilO(n^{\lceil k/2\rceil})$}{$\sim k\cdot n^{\lceil k/2\rceil}\sim n^{\lceil k/2\rceil}$}.

For a knot $K$ with $n$ crossings, let $K$ be represented as a Gauss diagram with $n$ arrows. 
By definition, $\varphi_k(K)$ is the sum of all renumbered subdiagrams with exactly $k$ arrows,
\[ \varphi_k(K)=\sum_{D \subset K,\ |D|=k} \psi(D),\]
where $D\subset K$ means $D$ is a subdiagram of $K$. 

Fixing a choice of $e,f\in\mathbb{Z}_{\geq 0}$ with $e+f=k$, every $k$-arrow subdiagram $D$ can be viewed in $\binom{k}{e}$ ways as the superimposition of two smaller subdiagrams $E$ and $F$ of sizes $e$ and $f$ respectively. Note that \ifugly{$e,f,k=\tilO(1)$}{$e,f,k\sim 1$}.
As discussed above, rather than breaking a specific $k$-arrow subdiagram down as a superimposition, one can instead build up all $k$-arrow subdiagrams by first choosing an $e$-arrow subdiagram $E$ and then choosing an $f$-arrow subdiagram $F$ that lies in the complement of $E$ in $K$.
Summing over the possible choices of $E$, $F$, and superimpositions gives the next formula\footnote{We note that all the scalars in this formula and in the formulas that follow are integers of magnitude \ifugly{$\tilO(n^k)$}{at most $\sim n^k$}. Additions of such integers take time $\ifug{1}$, and hence we need not worry about the time cost of working with large integers.}:

\begin{equation}\label{eqn:EFformula}  \varphi_k(K)=\sum_{
\small \begin{tabular}{c} $D \subset K$\\ $\left|D\right|=k$\end{tabular}} \psi(D)= \binom{k}{e}^{-1}\sum_{
\small \begin{tabular}{c} $E \subset K$\\ $\left|E\right|=e$\end{tabular}}
\sum_{\small\begin{tabular}{c}\text{patterns}\\$P$\end{tabular}} \sum_{
\small \begin{tabular}{c} $F \subset K, \; \left|F\right|=f$\\ $\pi_K(E,F)=P$\end{tabular}}\psi(E)\#_P \psi(F)
\end{equation}

\noindent The sum above overcounts every $k$-arrow subdiagram $D$ of $K$ by a factor of $\binom{k}{e}$ as every choice of splitting $D$ into two pieces (i.e. a choice of $E$)  occurs exactly once in the sum. Hence we include the pre-factor $\binom{k}{e}^{-1}$.

%\textcolor{red}{check $\lambda$ map and exponent here, should it be 2f-1 or 2f? Are we starting at 0 or 1?}

Define $\theta_K:([0,2n-1]_\Z)^f\rightarrow \mathcal{GD}_{f}$ by

  \[  (F_0,F_1,\cdots, F_{2f-1})\mapsto\begin{cases}
  \psi(F) & \text{ if } (F_0,F_1,\cdots, F_{2f-1}) \text{ are the ends of a subdiagram } F\subset K\\
  0 & \text{otherwise}
\end{cases}
\]

Now, the innermost sum from Equation (\ref{eqn:EFformula}) can be rewritten as a sum of values of $\theta_K$ as follows:
\begin{equation}\label{eq:doubledag}
\sum_{
\small \begin{tabular}{c} $F \subset K, \; \left|F\right|=f$\\ $\pi_K(E,F)=P$\end{tabular}}\psi(E)\#_\lambda \psi(F) 
\text{ }=\text{ }
\psi(E)\#_P \left ( \sum_{\small \prod_{i=0}^{2f-1}[m^-_{i,E,P},m^+_{i,E,P}]}\theta_K \right ),
\end{equation}
where $m^-_{i,E,P}$ is the minimal possible place for the $i$'th arrow-end of a Gauss diagram $F$ within $[0,2n-1]_\Z$ given that $E$ along with $F$ should make the pattern $P$, and $m^+_{i,E,P}$ is the maximal such value. Both of these quantities are easily determined by $E$ and $P$, though we spare the reader the explicit formulas.

The upshot is that the sum on the right can be computed very quickly with a look-up table.  Proposition~\ref{prop:table2} applies by taking $Q$ to be the set of all $(F_0,F_1,\cdots, F_{2f-1})\in ([0,2n-1]_\Z)^f$ that are the endpoints of a subdiagram of $K$. Here, $|Q|=\binom{n}{f}\ifugly{=\tilO(n^f)}{\sim n^f}$ and the conclusion from Proposition~\ref{prop:table2} is that a look-up table can be created in time $\ifug{n^f}$ so that computing the sum on the right hand side of Equation (\ref{eq:doubledag}) takes time $\ifug{1}$.
 
Thus we arrive at our final equation,

\begin{equation}\label{eqn:Finalformula} \varphi_k(K)=\binom{k}{e}^{-1} \sum_{
\small \begin{tabular}{c} $E \subset K$\\ $\left|E\right|=e$\end{tabular}}\sum_{\small\begin{tabular}{c}\text{patterns}\\$P$\end{tabular}} \psi(E)\#_P \left ( \sum_{\small \prod_{i=0}^{2f-1}[m^-_{i,E,P},m^+_{i,E,P}]}\theta_K \right ).
\end{equation}

To understand the computational complexity of Equation (\ref{eqn:Finalformula}), we need to understand the complexity of each sum. We already showed that the innermost sum can be computed in time $\ifug{n^f}$ by building a look-up table from Proposition~\ref{prop:table2}.
For the middle sum, $P$ can be any subset of $[0,2(e+f))_\Z$ of size $2e$. There are $\binom{2(e+f)}{2e} = \binom{2k}{2e}\ifugly{=\tilO(1)}{\sim 1}$ such subsets. That does not add to the complexity of the total sum.

For the fixed choice of $e+f=k$, the outer sum in Equation (\ref{eqn:Finalformula}) has at most \ifugly{$\tilO\left(\binom{n}{e}\right)=\tilO(n^e)$}{$\sim\binom{n}{e}\sim n^e$} terms.  Therefore, the total time of computing $\varphi_k(K)$, which includes first creating the look-up table and then computing the sum over all subdiagrams $D \subseteq K$ of size $k$, is $\ifug{n^e +n^f}$. 
This sum can be minimized for $e=\lceil \frac{k}{2} \rceil$ and $f=\lfloor \frac{k}{2} \rfloor$, in which case we get computation time $\ifug{n^{\lceil \frac{k}{2} \rceil}}$.
\end{proof}

Note that the lookup table within the proof of the main theorem is of size $\ifug{n^f}$, where at the end, $f$ is set to be $\lfloor\frac{k}{2}\rfloor$, and so our algorithm uses storage space $\ifug{n^{\lfloor\frac{k}{2}\rfloor}}$. If storage space is limited we can set $f$ differently and get the following:\\

\noindent \textbf{Main Theorem} (space-aware version). For any integer $f\leq\frac{k}{2}$, any finite type invariant of type $k$ can be computed on an $n$-crossing knot in time $\ifug{n^{k-f}}$ using storage space \ifugly{$\tilO(n^f)$.\newline\null\hfill$\Box$}{at most $\sim n^f$.\hfill$\Box$}

\bibliographystyle{alpha}
\bibliography{efficiently}
%\draftcut
%\printbibliography

\end{document}